\documentclass[12pt]{amsart}
\usepackage{amsmath,amssymb,amsthm,amscd}

\usepackage{graphicx}

\newtheorem{formula}{}[section]
\newtheorem{proposition}[formula]{Proposition}
\newtheorem{corollary}[formula]{Corollary}
\newtheorem{lemma}[formula]{Lemma}
\newtheorem{theorem}[formula]{Theorem}

\theoremstyle{definition}
\newtheorem{definition}[formula]{Definition}

\newtheorem{example}[formula]{Example}

\theoremstyle{remark}
\newtheorem{remark}[formula]{Remark}

\newcommand{\bs}[1]{\boldsymbol{#1}}
\newcommand{\wt}[1]{\widetilde{#1}}
\newcommand{\w}[1]{\widehat{#1}}

\begin{document}

\title[On uniqueness of the equivariant smooth structure on $\mathbb{R}\mathcal{Z}_P$]{On uniqueness of the equivariant smooth structure on~a~real moment-angle manifold}
\author[N.Yu.~Erokhovets, E.S.~Erokhovets]{Nikolai~Erokhovets and Elena Erokhovets}
\address{Department of Mechanics and Mathematics, Lomonosov Moscow State University\&International Laboratory of Algebraic Topology and its Applications, National Research University Higher School of Economics, Moscow}
\email{erochovetsn@hotmail.com}

\def\sgn{\mathrm{sgn}\,}
\def\bideg{\mathrm{bideg}\,}
\def\tdeg{\mathrm{tdeg}\,}
\def\sdeg{\mathrm{sdeg}\,}
\def\grad{\mathrm{grad}\,}
\def\ch{\mathrm{ch}\,}
\def\sh{\mathrm{sh}\,}
\def\th{\mathrm{th}\,}

\def\mod{\mathrm{mod}\,}
\def\In{\mathrm{In}\,}
\def\Im{\mathrm{Im}\,}
\def\Ker{\mathrm{Ker}\,}
\def\Hom{\mathrm{Hom}\,}
\def\Tor{\mathrm{Tor}\,}
\def\rk{\mathrm{rk}\,}
\def\codim{\mathrm{codim}\,}

\def\ko{{\mathbf k}}
\def\sk{\mathrm{sk}\,}
\def\RC{\mathrm{RC}\,}
\def\gr{\mathrm{gr}\,}

\def\R{{\mathbb R}}
\def\C{{\mathbb C}}
\def\Z{{\mathbb Z}}
\def\A{{\mathcal A}}
\def\B{{\mathcal B}}
\def\K{{\mathcal K}}
\def\M{{\mathcal M}}
\def\N{{\mathcal N}}
\def\E{{\mathcal E}}
\def\G{{\mathcal G}}
\def\D{{\mathcal D}}
\def\F{{\mathcal F}}
\def\L{{\mathcal L}}
\def\V{{\mathcal V}}
\def\H{{\mathcal H}}

\thanks{The work was carried out as a result of the research in the framework of the ``Mirror Laboratories'' project of the National Research University Higher School of Economics }





\begin{abstract}
The paper is devoted to the well-known problem of smooth structures
on moment-angle manifolds.
Each real or complex moment-angle manifold has an equivariant smooth
structure given by an intersection of quadrics corresponding to a geometric realisation
of a polytope. In 2006 F.~Bosio and L.~Meersseman proved that 
complex moment-angle manifolds of combinatorially equivalent simple polytopes
are equivariantly diffeomorphic. Using arguments from calculus we derive from this result that real moment-angle manifolds 
of combinatorially equivalent simple polytopes are equivariantly diffeomorphic and the
polytopes are diffeomorphic as manifolds with corners. 
\end{abstract}
\maketitle
\tableofcontents
\setcounter{section}{0}
\section{Introduction}
A moment-angle manifold $\mathcal{Z}_P$ of a simple $n$-polytope $P$ was introduced M.W.~Davis and T.~Januszkiewicz 
in \cite{DJ91}. If $P$ has $m$ facets, then by definition  $\mathcal{Z}_P$ is a~topological 
manifold with a canonical action of the torus $T^m=(S^1)^m$ with $\mathcal{Z}_P/T^m=P$. Moreover, if the polytopes $P$
and $Q$ are combinatorially equivalent, then $\mathcal{Z}_P$ and $\mathcal{Z}_Q$ are equivariantly homeomorphic.
It is mentioned in \cite[Section 6.4]{DJ91} that $\mathcal{Z}_P$ is a smooth manifold, but without details.
An explicit equivariant smooth structure was described by V.M.~Buchstaber and T.E.~Panov in \cite{BP02}, where   
the moment-angle manifold was identified with a certain intersection of real quadrics defined by a geometric realization of 
simple polytopes (see \cite[Construction 3.1.8]{BP02}, and also \cite[Section 3]{BPR07}). In \cite[Corollary 4.7]{BM06}
it was proved that for two geometric realizations of the same combinatorial simple
polytope the corresponding smooth moment-angle manifolds are equivariantly diffeomorphic. It is mentioned in 
\cite[Remark after Proposition 6.2.3]{BP15} that this  can be also proved using the fact that combinatorially
equivalent simple polytopes are diffeomorphic as manifolds with corners (see also \cite[Corollary 6.3]{D14} ). The
latter fact was proved by M.~Wiemeler in \cite[Corollary 5.3]{W13} (see also \cite[Corollary 1.3]{D14}). 

This article is devoted to the analogous question for the real moment-angle manifold $\mathbb R\mathcal{Z}_P$.
It then it can be identified with the set of real points of $\mathcal{Z}_P$ in its realization by quadrics and 
also is defined by quadrics. There is a natural action of $\mathbb Z_2^m$ on $\mathbb R\mathcal{Z}_P$. 
We give a survey of known results and give a new prove of the following fact.
\begin{theorem}
If the polytopes $P$ and $Q$ are combinatorially equivalent, then $\mathbb R\mathcal{Z}_P$ and $\mathbb R\mathcal{Z}_Q$ are equivariantly diffeomporphic.
\end{theorem}

The arguments from \cite{BM06} can not be directly generalized to real moment-angle manifolds, since
\cite[Lemma 3.4]{BM06} stating that any diffeomorphism of the standard sphere $S^{2k-1}\subset \mathbb C^k$ 
equivariant with respect to the standard action of $\mathbb T^k$ is equivariantly isotopic to identity is not valid for
$\mathbb Z_2$-equivariant diffeomorphisms of $S^k\subset\mathbb R^k$. In \cite[Section 3]{KMY18} S.~Kuroki, M.~Masuda, L.~Yu give a sketch of a proof that there  is a unique smooth structure on $\mathbb R\mathcal{Z}_P$ such that the
canonical $\mathbb Z_2^m$ action is smooth (as well as a unique smooth structure on a small cover such
that the canonical $\mathbb Z_2^n$-action is smooth). It follows the approach of M.~Wiemeler \cite{W13} 
and uses the notions of $G$-normal and $B$-normal systems introduced and studied by M.W.~Davis in~\cite{D78, D14}.

In our paper, starting from a $\mathbb T^m$-equivariant diffeomorphism $\mathcal{Z}_P\to \mathcal{Z}_Q$
constructed in \cite{BM06} we built a $\mathbb Z_2^m$-equivariant diffeomorphism 
$\mathbb R\mathcal{Z}_P\to\mathbb R\mathcal{Z}_Q$ and a diffeomorphism $P\to Q$ of manifolds with corners
with natural linear smooth structures. 
The latter uses the fact that the linear smooth structure on a simple polytope $P$ arising from its embedding
to $\mathbb R^n$ is diffeomorphic to the quadratic smooth structure arising from the real moment-angle manifold
$\mathbb R\mathcal{Z}_P$.

\section{Moment-angle manifold of a simple polytope}\label{sec:ma}
In this section we give basic definitions on moment-angle manifolds.

A {\em convex polytope} $P$ is a {\em bounded} intersection of a finite set of half-spaces:
$$
P=\{\boldsymbol{x}\in \mathbb R^n\colon \boldsymbol{a}_i\boldsymbol{x}+b_i\geqslant 0, i=1,\dots,m\}, 
$$
where $\boldsymbol{a}_i\in(\mathbb R^n)^*$ and $b_i\in\mathbb R$. 
For the $m\times n$-matrix $A$ with rows $\boldsymbol{a}_i$ its rank is equal to $n$, hence 
the affine mapping $j_P\colon \mathbb R^n\to \mathbb R^m$ defined as 
$$
j_P(\boldsymbol{x})=(y_1,\ldots,y_m), \text{where $y_i=\boldsymbol{a}_i\boldsymbol{x} + b_i$},
$$ 
is an embedding, and the image of the polytope $P$ is 
$$
j_P(\mathbb R^n)\cap \mathbb R^m_{\geqslant},\text{ where }\mathbb R^m_{\geqslant}=\{(y_1,\ldots,y_m)\in\mathbb R^m\colon y_i\geqslant 0\;\forall i\}.
$$
Take an $((m-n)\times m)$-matrix $C$ such that $CA=0$ and ${\rm rank\,}C=m-n$. 
Let $\boldsymbol{c}_i=(c_{1,i},\dots,c_{m-n,i})$. Then
$$
j_P(P)=\{(y_1,\dots,y_m)\in \mathbb R^m_{\geqslant}\colon c_{i,1}y_1+\dots+c_{i,m}y_m=c_i, 
i=1,\dots,m-n\},
$$
where $c_i=c_{i,1}b_1+\dots+c_{i,m}b_m$. 

\begin{example}
For the standard simplex 
$$
\Delta^n=\{x_i\geqslant 0\text{ for } i=1,\dots,n,\quad\text{ and }-x_1-\dots-x_n+1\geqslant 0\}
$$
we have $C=(1,\dots,1)$, and  
$$
j_{\Delta^n}(\Delta^n)=\{(y_1,\dots, y_{n+1})\in\mathbb R^{n+1}_{\geqslant}\colon y_1+\dots+y_{n+1}=1\}
$$
is a {\it regular simplex}\index{regular simplex}\index{simplex!regular}.
\end{example}

\begin{definition}
A {\em moment-angle manifold} is the subset in $\mathbb{C}^m$ defined as 
$$
\mathcal{Z}_P=\rho^{-1}\circ j_P(P), \text{where }\rho(z_1, \ldots, z_m) = \big(|z_1|^2, \ldots, |z_m|^2\big).
$$ 
The action of $\mathbb T^m$ on $\mathbb {C}^m$ induces the action of $\mathbb T^m$ on $\mathcal{Z}_P$.

A {\em real moment-angle manifold}
is the set of real points in $\mathcal{Z}_P$, that is a subset in $\mathbb{R}^m$ defined as 
$$
\mathbb R\mathcal{Z}_P=\rho^{-1}\circ j_P(P), \text{where }\rho(x_1, \ldots, x_m) = \big(|x_1|^2, \ldots, |x_m|^2\big).
$$ 
The action of $\mathbb Z_2^m$ on $\mathbb {C}^m$ by changing the signs of the coordinates induces the action of 
$\mathbb Z_2^m$ on $\mathbb R\mathcal{Z}_P$.
\end{definition}

For the embeddings $j_{\mathcal{Z}_P}\colon(\mathcal{Z}_P,\mathbb{R}\mathcal{Z}_P)
\subset (\mathbb C^m,\mathbb R^m)$ and $j_P\colon P\subset\mathbb R^m_{\geqslant}$ we have the commutative diagram:
$$ 
\begin{CD}
(\mathcal{Z}_P,\mathbb R\mathcal{Z}_P) @ >{j_{\mathcal{Z}_P}}>> (\mathbb{C}^m,\mathbb R^m)\\
@ V{\rho_P}VV @ VV{\rho}V\\
(P,P) @>{j_P}>> (\mathbb{R}_{\geqslant}^m,\mathbb{R}_{\geqslant}^m)
\end{CD}
$$

\begin{example}
For the standard simplex $\Delta^n$ the moment angle manifolds are spheres:
\begin{gather*}
\mathcal{Z}_{\Delta^n}=\{(z_1,\dots,z_{n+1})\in\mathbb C^{n+1}\colon |z_1|^2+\dots+|z_{n+1}|^2=1\}\simeq S^{2n+1};\\
\mathbb R\mathcal{Z}_{\Delta^n}=\{(u_1,\dots,u_{n+1})\in\mathbb R^{n+1}\colon |u_1|^2+\dots+|u_{n+1}|^2=1\}\simeq S^n.
\end{gather*}
\end{example}

Let $z_j=u_j+iv_j$.
\begin{theorem}(see \cite[Theorem 6.1.5.]{BP15})
\begin{enumerate}
\item We have $\mathcal{Z}_P=\{\boldsymbol{z}\in\mathbb C^m\colon c_{i,1}|z_1|^2+\dots+c_{i,m}|z_m|^2=c_i, i=1,\dots,m-n\}$,\\
$\mathbb R\mathcal{Z}_P=\{\boldsymbol{u}\in\mathbb R^m\colon c_{i,1}|u_1|^2+\dots+c_{i,m}|u_m|^2=c_i, 
i=1,\dots,m-n\}$. 
\item Denote $\Phi_k(\boldsymbol{z})=c_{k,1}|z_1|^2+\dots+c_{k,m}|z_m|^2-c_k$. Then at each point 
$\boldsymbol{u}\in \mathbb R\mathcal{Z}_P$ the differentials $d\Phi_1$, $\dots$, $d\Phi_{m-n}$ 
are linearly independent. Thus, 
\begin{itemize}
\item $\mathbb R\mathcal{Z}_P$  is a smooth manifold with a fixed trivialisation 
of the normal bundle of the $\mathbb Z_2^m$-equivariant embedding $\mathbb R\mathcal{Z}_P \subset \mathbb{R}^m$, 
and 
\item $\mathcal{Z}_P$ is a smooth manifold with a fixed trivialisation of the normal bundle 
of the $\mathbb T^m$-equivariant embedding $\mathcal{Z}_P \subset \mathbb{C}^m$.
\end{itemize}
\end{enumerate}
\end{theorem}

\begin{example}[A smooth atlas on a moment-angle manifold] \index{moment-angle manifold!smooth atlas}We can explicitly describe a smooth atlas on the manifold $\mathbb R\mathcal{Z}_P$ or 
$\mathcal{Z}_P$. For each vertex $v=F_{i_1}\cap \dots\cap F_{i_n}$, $i_1<\dots<i_n$, of~$P$ set 
$\omega_v=\{i_1,\dots, i_n\}$ and 
$$
P_v=P\setminus\bigcup\limits_{j\notin \omega_v}F_j.
$$
The set $P_v$ is homeomorphic to an open convex set in $\mathbb R^n_{\geqslant}$ with coordinates
$y_i=\boldsymbol{a}_i\boldsymbol{x}+b_i$, $i\in\omega_v$. This set is defined by linear inequalities 
$$
\boldsymbol{a}_jA_v^{-1}(\boldsymbol{y}_v-\boldsymbol{b}_v)+b_j=\widetilde{\boldsymbol{a}}^v_j\boldsymbol{y}_v+
\widetilde{b}^v_j> 0, j\notin \omega_v,
$$
where $\boldsymbol{y}_v=\begin{pmatrix}y_{i_1}\\ 
\dots\\
y_{i_n}
\end{pmatrix}$, $A_v=\begin{pmatrix}a_{i_1,1}&\dots&a_{i_1,n}\\ 
\dots\\
a_{i_n,1}&\dots&a_{i_n,n}
\end{pmatrix}$, $\boldsymbol{b}_v=\begin{pmatrix}b_{i_1}\\ 
\dots\\
b_{i_n}
\end{pmatrix}$, and $\boldsymbol{y}_v=A_v\boldsymbol{x}+\boldsymbol{b}_v$.
There is a homeomorphism  $P_v\simeq\mathbb R^n_{\geqslant}$.

The preimage $U_v=\rho_P^{-1}(P_v)$ in $\mathcal{Z}_P$ is an open subset 
defined by the inequalities $z_j\ne 0$, $j\notin \{i_1,\dots, i_n\}$. Moreover, 
$|z_j|^2=\widetilde{a}^v_{j,1}|z_{i_1}|^2+\dots + \widetilde{a}^v_{j,n}|z_{i_n}|^2+\widetilde{b}^v_j$. Hence, $U_v\simeq S_v\times\mathbb T^{m-n}$, 
where 
\begin{equation}\label{def:Sv}
S_v=\{(z_{i_1},\dots, z_{i_n})\in\mathbb C^n\colon  \widetilde{a}^v_{j,1}|z_{i_1}|^2+\dots + \widetilde{a}^v_{j,n}|z_{i_n}|^2+\widetilde{b}^v_j> 0, j\notin\omega_v\}
\end{equation}
 This is an~open set in $\mathbb C^n$ containing $\boldsymbol{0}$. Moreover,  
$S_v\simeq \mathbb R^n_{\geqslant}\times\mathbb T^n/\sim\simeq \mathbb C^n$.
For $z_j\ne 0$ set $\varphi_j={\rm arg}(z_j)$. Then the functions 
$$
(z_{i_1},\dots, z_{i_n}, \varphi_1,\dots,\widehat{\varphi_{i_1}},\dots,\widehat{\varphi_{i_n}},\dots,\varphi_m)
$$ 
are coordinates on $U_v$.

For another vertex $w=F_{j_1}\cap\dots\cap F_{j_n}$ we have the coordinates 
$$
(z_{j_1},\dots, z_{j_n}, \varphi_1,\dots,\widehat{\varphi_{j_1}},\dots,\widehat{\varphi_{j_n}},\dots,\varphi_m)
$$
and the transition functions on $U_v\cap U_w$ are
\begin{gather*}
z_{j_s}=\begin{cases} z_{j_s},&\text{ if }j_s\in\omega_v;\\
\sqrt{\widetilde{a}^v_{j_s,1}|z_{i_1}|^2+\dots + \widetilde{a}^v_{j_s,n}|z_{i_n}|^2+\widetilde{b}^v_{j_s}}e^{i\varphi_{j_s}},
&j_s\notin\omega_v;
\end{cases}\\
\varphi_j=\begin{cases}
{\rm arg}(z_j),&\text{ if }j\in\omega_v;\\
\varphi_j,&\text{ if } j_s\notin\omega_v.
\end{cases}
\end{gather*}
The preimage $W_v=\rho_P^{-1}(P_v)$ in $\mathbb R\mathcal{Z}_P$ is an open subset 
defined by the inequalities $u_j\ne 0$, $j\notin\omega_v$. Moreover, 
$u_j^2=\widetilde{a}^v_{j,1}u_{i_1}^2+\dots + \widetilde{a}^v_{j,n}u_{i_n}^2+\widetilde{b}^v_j$. Hence, 
$\rho_P^{-1}(P_v)\simeq \mathbb RS_v\times \mathbb Z_2^{m-n}$, 
where 
$$
\mathbb R S_v=\{(u_{i_1},\dots, u_{i_n})\in\mathbb C^n\colon  \widetilde{a}^v_{j,1}u_{i_1}^2+\dots + \widetilde{a}^v_{j,n}
u_{i_n}^2+\widetilde{b}^v_j> 0, j\notin\omega_v\}
$$ 
This is an~open set in $\mathbb R^n$ containing $\boldsymbol{0}$. Moreover,  
$\mathbb R S_v\simeq \mathbb R^n_{\geqslant}\times\mathbb Z_2^n/\sim\simeq \mathbb R^n$.

For a given choice of signs $\varepsilon_j=\pm 1$, $j\notin \omega_v$
the functions $(u_{i_1},\dots, u_{i_n})$ are coordinates on the open set 
$$
W_{v,\varepsilon}=\rho_P^{-1}(P_v)\cap \{{\rm sign}(u_j)=\varepsilon_j\colon j\notin \omega_v\}\simeq \mathbb{R}S_v.
$$

For another vertex $w=F_{j_1}\cap\dots\cap F_{j_n}$ and signs $\{\eta_j\}$ there are the coordinates 
$(u_{j_1},\dots, u_{j_n})$ on $W_{w,\eta}$. We have $W_{v,\varepsilon}\cap W_{w,\eta}\ne\varnothing$
if and only if $\varepsilon_j=\eta_j$ for all $j\notin \omega_v\cup\omega_w$.
The transition functions are
$$
u_{j_s}=\begin{cases} u_{j_s},&\text{ if }j_s\in\omega_v;\\
\varepsilon_{j_s}\sqrt{\widetilde{a}^v_{j_s,1}u_{i_1}^2+\dots + \widetilde{a}^v_{j_s,n}u_{i_n}^2+\widetilde{b}^v_{j_s}},
&\text{ if }j_s\notin\omega_v.
\end{cases}
$$
\end{example}
\section{Quadratic and linear smooth structures on the polytope as a manifold with corners}
\begin{definition}[A manifold with corners] (see \cite[Definition 7.1.3]{BP15}, \cite{J12})
{\it A~manifold with corners}\index{manifold with corners} (of dimension $n$) is a topological
manifold $Q$ with boundary together with an atlas $\{U_i,\varphi_i\}$ consisting of homeomorphisms
$\varphi_i\colon U_i\to W_i$ onto open subsets $W_i\subset \mathbb R^n_{\geqslant}$
such that the mapping $\varphi_i\circ\varphi_j^{-1}\colon \varphi_j(U_i\cap U_j)\to \varphi_i(U_i\cap U_j)$
is smooth for all $i,j$. (A mapping between open
subsets in $\mathbb R^n_{\geqslant}$ is called smooth if can be obtained by restriction 
of a~smooth mapping of open subsets in $\mathbb R^n$.)
\end{definition}
\begin{example}[A quadratic smooth structure on a polytope]\label{ex:smpind}
\index{polytope!smooth structure as a manifold with corners}
The smooth atlases on $\mathcal{Z}_P$ and $\mathbb R\mathcal{Z}_P$ 
induce a smooth atlas on $P$ as a manifold with corners.
Its charts are $P_v$ and coordinates in $P_v$ are $u_j=\sqrt{\boldsymbol{a}_j\boldsymbol{x}+b_j}$, $j\in \omega_v$.
Then for two vertices $v$ and $w$ we have 
$$
u_{j_s}=\begin{cases} u_{j_s},&\text{ if }j_s\in\omega_v;\\
\sqrt{\widetilde{a}^v_{j_s,1}u_{i_1}^2+\dots + \widetilde{a}^v_{j_s,n}u_{i_n}^2+\widetilde{b}^v_{j_s}},
&\text{ if }j_s\notin\omega_v.
\end{cases}
$$
\end{example}
\begin{example}[A linear smooth structure on a polytope]\label{ex:smpnat}
On the other hand, there is natural smooth structure on $P$ as a manifold with corners.
Its charts are again $P_v$, and coordinates in $P_v$ are $y_j=\boldsymbol{a}_j\boldsymbol{x}+b_j$, $j\in \omega_v$.
Then for two vertices $v$ and $w$ we have 
$$
y_{j_s}=\begin{cases} y_{j_s},&\text{ if }j_s\in\omega_v;\\
\widetilde{a}^v_{j_s,1}y_{i_1}+\dots + \widetilde{a}^v_{j_s,n}y_{i_n}+\widetilde{b}^v_{j_s},
&\text{ if }j_s\notin\omega_v.
\end{cases}
$$
\end{example}
\begin{remark}
It was proved by M.~Wiemeler in~\cite{W13} (see also \cite{D14}) that if two simple polytopes $P$ and $Q$ 
are combinatorially equivalent, then they are diffeomorphic as manifolds with corners with linear smooth structures.
\end{remark}

\begin{proposition}\label{prop:lqss}
The polytope $P$ with linear smooth structure is diffeomorphic to $P$ 
with quadratic smooth structure.
\end{proposition}
We will give a sketch the proof suggested to the authors by A.A.~Gaifullin.
\begin{proof}[Sketch of a proof of Proposition \ref{prop:lqss}] 
\begin{definition}[\cite{G08}]
A {\it collar}\index{collar of a facet}\index{polytope!collar of a facet} 
of a facet $F_i$ of a simple convex polytope $P\subset \mathbb R^n$ is a smooth embedding 
$\varphi\colon F_i\times [0,\delta)\to P^n$ such that
\begin{enumerate}
\item $\varphi\left.\right|_{F_i\times 0}={\rm id}_{F_i}$ -- the identity mapping;
\item the image of $\varphi$ contains an open neighbourhood of $F_i$ in $P$ (in the topology induced from $\mathbb R^n$);
\item $\varphi((G\cap F_i)\times [0,\delta))\subset G$ for any face $G$ of $P$ such that $G\not\subset F_i$ 
but $G\cap F_i\ne\varnothing$.
\end{enumerate} 

For $0<\delta'<\delta$ the collar $\varphi$ can be restricted  to a collar $\varphi\left.\right|_{F_i\times[0,\delta')}$. 

For any face $G$ such that $G\not\subset F_i$ and $G\cap F_i\ne\varnothing$ 
the~embedding $\varphi_G=\varphi\left.\right|_{(G\cap F_i)\times[0,\delta)}$ is called 
an~{\it induced collar}\index{induced collar}\index{polytope!induced collar} .

For two different facets $F_i$ and $F_j$ with $F_i\cap F_j\ne\varnothing$ the corresponding collars
$\varphi_i$ and $\varphi_j$ are  {\it consistent} \index{consistent collars}\index{polytope!consistent collars} 
if the diagram
$$
\begin{CD}
  (F_i\cap F_j)\times [0,\delta)\times[0,\delta)@>(\varphi_i)_{F_j}\times {\rm id}>>F_j\times [0,\delta)@>\varphi_j>>P\\
  @V{{\rm id}\times T}VV@.@V{\rm id}VV\\
  (F_i\cap F_j)\times [0,\delta)\times[0,\delta)@>(\varphi_j)_{F_i}\times {\rm id}>> F_i\times[0,\delta)@>\varphi_i>>P
\end{CD} 
$$
where $T(t_1,t_2)=(t_2,t_1)$, is commutative (perhaps, after a restriction to a smaller half interval $[0,\delta')$). 
By definition collars of disjoint facets are consistent.

A {\it system of consistent collars}\index{system of consistent collars}\index{polytope!system of consistent collars} 
is a set of collars $\{\varphi_1,\dots,\varphi_m\}$ of all facets such that
any two collars are consistent.  
\end{definition}
\begin{proposition}[\cite{G08}]
Any simple polytope $P$ admits a system of consistent collars.
\end{proposition}
Given a system of consistent collars one can build a desired diffeomorphism as follows.
Take a diffeomorphism $\xi\colon [0,\infty)\to [0,\infty)$ such that 
$$
\xi(t)=\begin{cases} t^2&\text{ if }t<\frac{\delta}{10};\\
t&\text{ if }t>\frac{\delta}{2}.
\end{cases}
$$
Choosing small enough $\delta$ we can assume that if the images of any set of collars intersect
then the corresponding facets have a nonempty common intersection.
Now let a point $p\in P$ belong to the images of exactly $k$ collars $\varphi_{i_1}$, $\dots$, $\varphi_{i_k}$.
Then $G=F_{i_1}\cap\dots\cap F_{i_k}\ne\varnothing$. Consider the mapping
$\varphi_{i_1,\dots, i_k}\colon G\times [0,\delta)^k\to P$ defined as 
$$
\varphi_{i_1,\dots, i_k}=\varphi_{i_1}\circ\left((\varphi_{i_2})_{F_{i_1}}\times {\rm id}_{[0,\delta)}\right)
\circ\dots\circ\left((\varphi_{i_k})_{F_{i_1}\cap\dots\cap F_{i_{k-1}}}\times{\rm id}_{[0,\delta)^{k-1}}\right).
$$
Since the collars are consistent, the mapping $\varphi_{i_1,\dots, i_k}$ does not depend on the
order of the facets, that is 
$$
\varphi_{i_1,\dots, i_k}(\bs{x},t_k,\dots, t_1)=\varphi_{i_{\sigma(1)},\dots, i_{\sigma(k)}}(\boldsymbol{x}, 
t_{\sigma(k)},\dots, t_{\sigma(1)})
$$
for any permutation $\sigma$ of the set $\{1,\dots, k\}$.
Then 
$$p\in U_G=\varphi_{i_1,\dots, i_k}({\rm int} (G)\times [0,\delta)^k),
$$
for otherwise, 
$$
p\in \varphi_{i_1,\dots, i_k}((G\cap F_{i_{k+1}})\times [0,\delta)^k)\subset F_{i_{k+1}}\subset 
\varphi_{i_{k+1}}(F_{i_{k+1}}\times[0,\delta)).
$$ 
Then for any local coordinates $(z_1,\dots, z_{n-k})$ in ${\rm int}(G)$  
the functions $(z_1,\dots z_{n-k}, t_1,\dots, t_k)$ are local coordinates in the open set $U_G$. Define 
$$
F(z_1,\dots, z_{n-k}, t_1,\dots, t_k)=(z_1,\dots, z_{n-k},\xi(t_1),\dots, \xi(t_k)).
$$
It can be shown that $F$ is a~correctly defined mapping on $P$ and it is a diffeomorphism
of $P$ with linear smooth structure to $P$ with quadratic smooth structure.  
\end{proof}

\section{Uniqueness of equivariant smooth structures}
It was proved by F.~Bosio and L.~Meersseman in \cite[Theorem 4.1 and Corollary 4.7]{BM06} 
that if the geometric polytopes $P$ and $Q$ are combinatorially
equivalent, then the manifolds $\mathcal{Z}_P$  and $\mathcal{Z}_Q$ are $\mathbb T^m$-equivariantly
diffeomorphic. Thus, the defined equivariant smooth structure on the moment-angle complex $\mathcal{Z}_P$ does not
depend on~the~geometric realization~of~$P$.  

For real moment-angle manifolds the analogous argument
can not be followed literally due to the following reason. The proof uses \cite[Lemma 3.4]{BM06}: any diffeomorphism
$f$ of the sphere $S^{2a-1}=\{(z_1,\dots, z_a)\in \mathbb C^a\colon |z_1|^2+\dots+|z_a|^2=1\}$ equivariant with respect to
the standard action of $\mathbb T^a$ on $\mathbb C^a$ is equivariantly isotopic to identity. The analogous statement 
for the sphere $S^{a-1}=\{(w_1,\dots, w_a)\in\mathbb R^a\colon |w_1|^2+\dots+|w_a|^2=1\}$  equivariant with respect to
the standard action of~$\mathbb Z_2^a$ on $\mathbb R^a$ by changing signs of coordinates is not valid. The counterexample
is given for odd $a$ by the mapping $(w_1,\dots,w_a)\to(-w_1,\dots,-w_a)$. It has degree $-1$ and is not isotopic to identity.
We will build a $\mathbb Z_2^m$-equivariant diffeomorphism 
$\widehat{f}\colon \mathbb R\mathcal{Z}_P\to \mathbb R\mathcal{Z}_Q$ using the corresponding $\mathbb T^m$-equivariant 
diffeomorphism $f\colon \mathcal{Z}_P\to\mathcal{Z}_Q$

\begin{theorem}\label{RPth}
Let $f\colon\mathcal{Z}_P\to \mathcal{Z}_Q$ be 
a~$\mathbb T^m$-equivariant diffeomorphism. Then 
\begin{enumerate}
\item $f(R_1e^{i\varphi_1},\dots,R_me^{i\varphi_m})$ has the form 
$$
(e^{i\varphi_1}g_1(R_1,\dots, R_m)e^{i\psi_1(R_1,\dots, R_m)},\dots,e^{i\varphi_m}g_m(R_1,\dots, R_m)e^{i\psi_m(R_1,\dots, R_m)}),
$$
where $g_k=0$ if $R_k=0$, and  $g_k(R_1,\dots, R_m)\geqslant 0$, $\psi_k(R_1,\dots, R_m)$ are continuous functions on~$P$
considered as the image of the section $s_P\colon P\to \mathbb R\mathcal{Z}_P\subset 
\mathcal{Z}_P\subset\mathbb C^m$: 
$$
s_P(\boldsymbol{x})=(\sqrt{\boldsymbol{a}_1\boldsymbol{x}+b_1},\dots,\sqrt{\boldsymbol{a}_m\boldsymbol{x}+b_m})
$$
\item The mapping 
$$
g(R_1,\dots, R_m)=(g_1(R_1,\dots, R_m), \dots, g_m(R_1,\dots, R_m))
$$
is a diffeomorphism of smoothness at least $C^1$
between $P$ and $Q$ with quadratic smooth structures (Example~\ref{ex:smpind}). 
Moreover, it maps facets to facets and in local 
coordinates $R_{i_1},\dots, R_{i_n}$ at each vertex $v=F_{i_1}\cap\dots\cap F_{i_n}$ for functions
$\wt{g}_{i_s}(R_{i_1},\dots, R_{i_n})=g_{i_s}(R_1,\dots, R_m)$ we have $\frac{\partial \wt{g}_{i_s}}{\partial R_{i_t}}=0$
if $t\ne s$ and $R_{i_t}=0$.
\item The mapping
$$
\widehat{f}(u_1,\dots,u_m)=({\rm sign}(u_1)g_1(|u_1|,\dots, |u_m|),\dots, {\rm sign}(u_m)g_m(|u_1|,\dots, |u_m|))
$$
is a $\mathbb Z_2^m$-equivariant diffeomorphism $\mathbb R\mathcal{Z}_P\to \mathbb R\mathcal{Z}_Q$
of smoothness at~least~$C^1$.
\end{enumerate}
\end{theorem}
\begin{corollary}[\cite{W13}]
Combinatorially equivalent simple polytopes are diffeomorphic as manifolds with corners (with linear smooth structures).
\end{corollary}
\begin{proof}
This follows from the fact that linear and quadratic smooth structures are diffeomorphic.
\end{proof}
\begin{remark}
In \cite[Section 3]{KMY18} S.~Kuroki, M.~Masuda, L.~Yu give a sketch of a proof that there 
is a unique smooth structure on $\mathbb R\mathcal{Z}_P$ such that the
canonical $\mathbb Z_2^m$ action is smooth (as well as a unique smooth structure on a small cover such
that the canonical $\mathbb Z_2^n$-action is smooth). It follows the approach of M.~Wiemeler \cite{W13} 
and uses the notions of $G$-normal and $B$-normal systems introduced and studied by M.W.~Davis in~\cite{D78, D14}.
\end{remark}
\begin{proof}[Proof of Theorem~\ref{RPth}]
Let  $z_k=u_k+iv_k=R_ke^{i\varphi_k}$ for all $k=1,\dots,m$. 
Since $f$ is $\mathbb T^m$-equivariant, we have 
\begin{multline*}
f(R_1e^{i\varphi_1},\dots,R_m e^{i\varphi_m})=(f_1,\dots, f_m)=
(e^{i\varphi_1},\dots, e^{i\varphi_m})f(R_1,\dots, R_m)=\\
(e^{i\varphi_1}g_1(R_1,\dots, R_m)e^{i\psi_1(R_1,\dots, R_m)},\dots,
e^{i\varphi_m}g_m(R_1,\dots, R_m)e^{i\psi_m(R_1,\dots, R_m)}),
\end{multline*}
where $g_k(R_1,\dots,R_m)\geqslant 0$ are continuous functions on $P$
and each $\psi_k(R_1, \dots, R_m)$ is~a~continuous function on $P\setminus F_k$. The~mapping
$g(R_1,\dots, R_m)$ is a homeomorphism between the orbit spaces $P$ and $Q$.
Since $f$ preserves the stabilizers of points, $g$ maps each facet of $P$
defined by the equation $R_k=0$ to the facet of $Q$ defined by the~equation $g_k=0$. 
Then the~mapping $\widehat{f}\colon \mathbb R\mathcal{Z}_P\to\mathbb R\mathcal{Z}_Q$ 
is~a~$\mathbb Z_2^m$-equivariant
homeomorphism. We need to prove that $g$ and $\widehat{f}$ are smooth. 

For a subset $\tau=\{k_1,\dots,k_l\}\subset[m]$ and variables $\bs{X}=(X_1,\dots, X_m)$ define 
$\boldsymbol{X}_{\tau}=(X_{k_1},\dots, X_{k_l})$. For a vertex $v=F_{i_1}\cap\dots\cap F_{i_n}\in P$
define $\boldsymbol{X}_{v}=\boldsymbol{X}_{\omega_v}$,
$\bs{X}_{v\setminus j}=\bs{X}_{\omega_v\setminus\{j\}}$, $\bs{X}_{\widehat{v}}=\bs{X}_{[m]\setminus\omega_v}$,
where $\omega_v=\{i_1,\dots, i_n\}$.

Since polar coordinates are $C^{\infty}$-smooth on $\mathbb C\setminus\{0\}$, $g$ and 
$\widehat{f}$ are smooth over $P\setminus\partial P$. Consider a point $\bs{x}^0\in\partial P$.
Let $\omega_0=\{j\colon \bs{x}^0\in F_j\}$ and $G_0=\bigcap_{j\in \omega_0}F_j$. Take any 
vertex $v=F_{i_1}\cap\dots\cap F_{i_n}\in G_0$.
Any point $\bs{z}^0\in \rho_P^{-1}(\bs{x}^0)\in\mathcal{Z}_P$ belongs 
to~$U_v^P\simeq S_v\times \mathbb T^{m-n}$ (see (\ref{def:Sv})), where $z_{i_1}$, $\dots$, $z_{i_n}$,
$\varphi_j$, $j\notin \omega_v$, are smooth coordinates. Similarly, $f(\bs{z}^0)\in U_w^Q$, $w=g(v)$ 
(with $\omega_v=\omega_w$), where  $\wt{z}_{i_1}$, $\dots$, $\wt{z}_{i_n}$,
$\wt{\varphi}_j$, $j\notin \omega_v$ are smooth coordinates. In local coordinates $f$ has the form  
\begin{gather*}
\wt{z}_j=\wt{f}_j(\bs{z}_v)=\wt{g}_j(\bs{R}_v)e^{i(\varphi_j+\wt{\psi}_j(\bs{R}_v))},\;j\in\omega_v;\\
\wt{\varphi}_j=\varphi_j+\wt{\psi}_j(\bs{R}_v), \;j\notin \omega_v.
\end{gather*}

\begin{lemma}\label{RPl1}
The mapping $f_v\colon S_v^P\to S_w^Q$ (see (\ref{def:Sv})) defined as 
$$
f_v(\bs{z}_v)=(\wt{f}_{i_1}(\bs{z}_v),\dots, \wt{f}_{i_n}(\bs{z}_v))=(\wt{g}_{i_1}(\bs{R}_v)e^{i(\varphi_{i_1}+\wt{\psi}_{i_1}(\bs{R}_v))},
\dots, \wt{g}_{i_n}(\bs{R}_v)e^{i(\varphi_{i_n}+\wt{\psi}_{i_n}(\bs{R}_v))}),
$$
is a $\mathbb T^n$-equivariant diffeomorphism with respect to the standard action
of $\mathbb T^n$ on~$\mathbb C^n$.
\end{lemma}
\begin{proof}
Since $\wt{f}_j$, $j\in \omega_v$, does not depend on $\varphi_k$, $k\notin\omega_v$, and
the same holds for $\wt{f}^{-1}$, we see that $f_v$ and $f^{-1}_w$ are 
smooth mappings inverse to each other. So they are both diffeomorphisms. Moreover, by definition
they are $\mathbb T^n$-equivariant.
\end{proof}
\begin{lemma}\label{RPl2}
Let $f\colon A\to B$ be a $\mathbb T^n$-equivariant diffeomorphism of open $\mathbb T^n$-equivariant
neighbourhoods $A,B\subset \mathbb C^n$ of $\bs{0}$. Then 
\begin{enumerate}
\item $f(\bs{z})=(g_1(\bs{R})e^{i(\varphi_1+\psi_1(\bs{R}))},
\dots, 
g_n(\bs{R})e^{i(\varphi_n+\psi_n(\bs{R}))})$, 
where $g_k=0$ if $R_k=0$, and  $g_k(\bs{R})\geqslant 0$, $\psi_k(\bs{R})$ are continuous functions on
$A\cap\mathbb R^n_{\geqslant}$.
\item The mapping $g(\bs{R})=(g_1(\bs{R}), \dots, g_n(\bs{R}))$ is a homeomorphism  
between $A\cap \mathbb R^n_{\geqslant}$ and $B\cap \mathbb R^n_{\geqslant}$.
Its partial derivatives $\frac{\partial g_k}{\partial R_l}$ (when $R_l=0$ it is $\frac{\partial g_k}{\partial_+R_l}$) 
are defined and continuous, and $\frac{\partial g_k}{\partial_+R_l}=0$ when $k\ne l$ and $R_l=0$.
\item The mapping
$$
\widehat{f}(u_1,\dots,u_n)=({\rm sign}(u_1)g_1(|u_1|,\dots, |u_n|),\dots, {\rm sign}(u_n)g_n(|u_1|,\dots, |u_n|))
$$
is a $\mathbb Z_2^n$-equivariant diffeomorphism $A\cap\mathbb R^n\to B\cap \mathbb R^n$
of smoothness at~least~$C^1$.
\end{enumerate}
\end{lemma}
\begin{proof}
For a point $\bs{z}^0=(z^0_1,\dots,z^0_n)\in A$ set 
$$
\bs{u}^0=(u^0_1,\dots,u^0_n)={\rm Re}(\bs{z}^0),\; \bs{R}^0=(|z^0_1|,\dots, |z^0_n|), \text{ and }\omega_0=\{j\in[n]\colon z_j^0=0\}.
$$
For $j\notin \omega_0$ and $\varepsilon_j={\rm sign}(u^0_j)$ we have $f_j(\bs{u}^0)\ne 0$ and the~function  
$$
\widehat{f}_j(\bs{u})=\varepsilon_j g(|u_1|,\dots, |u_n|)=\varepsilon_j|f_j(u_1,\dots,u_n)|
\colon A\cap \mathbb R^n\to\mathbb R
$$ 
is smooth at $\bs{u}^0$. Similarly, $g_j(\bs{R})=|f_j(R_1,\dots,R_n)|\colon A\cap \mathbb R^n_{\geqslant}\to\mathbb  R$ 
is smooth at~$\bs{R}^0$, where we consider $\frac{\partial g_j(\bs{R})}{\partial_+R_s}$ if $R_s=0$.

Now consider $j\in \omega_0$. Let  $z_k=u_k+iv_k=R_ke^{i\varphi_k}$ for all $k=1,\dots, n$. 
The function $f_j(z_1,\dots, z_n)\colon A\to\mathbb R$ is smooth at $\bs{z}^0$ and $f_j(\bs{z}^0)=0$.  

Since $f_j(\bs{z})=0$ if $z_j=z^0_j=0$, we have 
\begin{equation}\label{pfk0}
\frac{\partial f_j(\bs{z}^0)}{\partial u_k}=\frac{\partial f_j(\bs{z}^0)}{\partial v_k}=0\text{ for }k\ne j.
\end{equation}
Similarly, $\frac{\partial g_j(\bs{R}^0)}{\partial R_k}=0$ for $k\notin \omega_0$ and
$\frac{\partial g_j(\bs{R}^0)}{\partial_+R_k}=0$ for $k\in \omega_0\setminus\{j\}$.

Now let us calculate $\frac{\partial f_j(\bs{z}^0)}{\partial u_j}$ and $\frac{\partial f_j(\bs{z}^0)}{\partial v_j}$. We have 
\begin{gather*}
\frac{\partial f_j(\bs{z}^0)}{\partial u_j}=
\lim_{t\to 0}\frac{g_j(R^0_1,\dots, R^0_{j-1},|t|, R^0_{j+1},\dots, R^0_n){\rm sign}(t)e^{i \psi_j(R^0_1,\dots, |t|, \dots, R^0_n)}}{t}=\\
=\lim_{t\to 0}\frac{g_j(R^0_1,\dots, |t|, \dots, R^0_n)e^{i \psi_j(R^0_1,\dots, |t|, \dots, R^0_n)}}{|t|};\\
\frac{\partial f_j(\bs{z}^0)}{\partial v_j}=
\lim_{t\to 0}i\frac{g_j(R^0_1,\dots, |t|, \dots, R^0_n)e^{i \psi_j(R^0_1,\dots, |t|, \dots, R^0_n)}}{|t|}.
\end{gather*}
Moreover, $\frac{\partial f_j(\bs{z}^0)}{\partial u_j}\ne \bs{0}$ and $\frac{\partial f_j(\bs{z}^0)}{\partial v_j}\ne\bs{0}$,
since $df(\bs{z}^0)$ is non-degenerate and equalities (\ref{pfk0}) hold. Then 
the limits $\frac{\partial g_j(\bs{R}^0)}{\partial_+R_j}=
\lim_{t\to 0^+}\frac{g_j(R^0_1,\dots, t, \dots, R^0_n)}{t}>0$
and $\psi_j(\bs{R}^0)=\lim_{t\to 0^+} \psi_j(R^0_1,\dots, t, \dots, R^0_n)$ exist
and 
$$
\begin{pmatrix}
\frac{\partial f_j(\bs{z}^0)}{\partial u_j}&\frac{\partial f_j(\bs{z}^0)}{\partial v_j}\end{pmatrix}=
\begin{pmatrix}\frac{\partial g_j(\bs{R}^0)}{\partial_+R_j}e^{i \psi_j(\bs{R}^0)}&
i\frac{\partial g_j(\bs{R}^0)}{\partial_+R_j}e^{i\psi_j(\bs{R}^0)}\end{pmatrix}.
$$
This argument shows that for $\bs{R}\in A\cap\mathbb R^n_{\geqslant}$ with $R_j=0$ we have  
$$
\frac{\partial g_j(\bs{R})}{\partial_+R_j}=\left|\frac{\partial f_j(\bs{R})}{\partial u_j}\right|=
\left|\frac{\partial f_j(\bs{R})}{\partial v_j}\right|>0\text{ and }
\psi_j(\bs{R})={\rm arg}\left(\frac{\partial f_j(\bs{R})}{\partial u_j}\right).
$$
In particular,  
$$
\frac{\partial g_j(\bs{R})}{\partial_+R_j}\to \frac{\partial g_j(\bs{R}^0)}{\partial_+R_j},\text{ and } \psi_j(\bs{R})\to\psi_j(\bs{R}^0),
\text{ when }\bs{R}\to \bs{R}^0\text{ and }R_j=0.
$$

For $z\ne0$ the functions $R$ and $\varphi$ are smooth coordinates on~$\mathbb C$,
and for any smooth function $h\colon \mathbb C\to\mathbb C$ we have 
\begin{gather*}
\begin{pmatrix}\frac{\partial h}{\partial u}&\frac{\partial h}{\partial v}\end{pmatrix}=
\begin{pmatrix}\frac{\partial h}{\partial R}&\frac{\partial h}{\partial \varphi}\end{pmatrix}
\begin{pmatrix}\frac{\partial R}{\partial u}&\frac{\partial R}{\partial v}\\
\frac{\partial \varphi}{\partial u}&\frac{\partial \varphi}{\partial v}
\end{pmatrix}=
\begin{pmatrix}\frac{\partial h}{\partial R}&\frac{\partial h}{\partial \varphi}\end{pmatrix}
\begin{pmatrix}\frac{u}{R}&\frac{v}{R}\\
\frac{-v}{R^2}&\frac{u}{R^2}
\end{pmatrix}=\\
\frac{1}{R}\begin{pmatrix}u\frac{\partial h}{\partial R}-\frac{v}{R}\frac{\partial h}{\partial \varphi}&
v\frac{\partial h}{\partial R}+\frac{u}{R}\frac{\partial h}{\partial \varphi}\end{pmatrix}=
\begin{pmatrix}\cos\varphi\frac{\partial h}{\partial R}-\frac{\sin\varphi}{R}\frac{\partial h}{\partial \varphi}&
\sin\varphi\frac{\partial h}{\partial R}+\frac{\cos\varphi}{R}\frac{\partial h}{\partial \varphi}\end{pmatrix}.
\end{gather*}
In particular, for $R_j\ne 0$ we obtain 
\begin{gather*}
\frac{\partial f_j}{\partial u_j}=\cos\varphi_j\left(\frac{\partial g_j}{\partial R_j}+
i g_j\frac{\partial \psi_j}{\partial R_j}\right)e^{i(\varphi_j+\psi_j)}-
i\frac{\sin\varphi_j}{R_j}g_je^{i(\varphi_j+\psi_j)};\\
\frac{\partial f_j}{\partial v_j}=\sin\varphi_j\left(\frac{\partial g_j}{\partial R_j}+
i g_j\frac{\partial \psi_j}{\partial R_j}\right)e^{i(\varphi_j+\psi_j)}+
i\frac{\cos\varphi_j}{R_j}g_je^{i(\varphi_j+\psi_j)}.
\end{gather*}
We have $\begin{pmatrix}\frac{\partial f_j}{\partial u_j}&\frac{\partial f_j}{\partial v_j}\end{pmatrix}\to
\begin{pmatrix}\frac{\partial f_j}{\partial u_j}(\bs{z}^0)&\frac{\partial f_j}{\partial v_j}(\bs{z}^0)\end{pmatrix}$ 
when $\bs{z}\to\bs{z}^0$.  Now~put $\varphi_1=0$, $\dots$, $\varphi_n=0$.
Under this condition $\bs{z}$ is equal to $\bs{R}$ and tends to~$\bs{R}^0$, and  
\begin{gather*}
\begin{pmatrix}\frac{\partial f_j}{\partial u_j}&\frac{\partial f_j}{\partial v_j}\end{pmatrix}=
\begin{pmatrix}(\frac{\partial g_j}{\partial R_j}+
i g_j\frac{\partial \psi_j}{\partial R_j})e^{i \psi_j}
&
i\frac{g_j}{R_j}e^{i\psi_j}
\end{pmatrix}\to\begin{pmatrix}\frac{\partial g_j(\bs{R}^0)}{\partial_+R_j}e^{i \psi_j(\bs{R}^0)}&
i\frac{\partial g_j(\bs{R}^0)}{\partial_+R_j}e^{i \psi_j(\bs{R}^0)}\end{pmatrix}
\end{gather*}
From the second component we have 
$$
\frac{g_j(\bs{R})}{R_j}\to\frac{\partial g_j(\bs{R}^0)}{\partial_+R_j}
\text{ and }\psi_j(\bs{R})\to\psi_j(\bs{R}^0)\text{ when }\bs{R}\to\bs{R}^0\text{ and }R_j\ne 0.
$$ 
Now from the first component 
$$
\frac{\partial g_j(\bs{R})}{\partial R_j}+
i g_j(\bs{R})\frac{\partial \psi_j(\bs{R})}{\partial R_j}\to\frac{\partial g_j(\bs{R}^0)}{\partial_+R_j} 
\text{ when }\bs{R}\to\bs{R}^0\text{ and }R_j\ne 0.
$$
Hence, 
$$
\frac{\partial g_j(\bs{R})}{\partial R_j}\to\frac{\partial g_j(\bs{R}^0)}{\partial_+R_j} \text{ and }
g_j(\bs{R})\frac{\partial \psi_j(\bs{R})}{\partial R_j}\to\bs{0}
\text{ when  }\bs{R}\to\bs{R}^0\text{ and }R_j\ne 0.
$$

Thus, at the point $\bs{R}^0\in  A\cap \mathbb{R}^n_{\geqslant}$ the partial derivative 
$\frac{\partial g_j(\bs{R}^0)}{\partial_+R_j}$ is defined and is~equal to the limit
of partial derivatives $\frac{\partial g_j(\bs{R})}{\partial R_j}$ when $\bs{R}\to \bs{R}^0$ and $R_j\ne 0$
and  $\frac{\partial g_j(\bs{R})}{\partial_+ R_j}$ when $\bs{R}\to \bs{R}^0$ and $R_j=0$.

This proves that $\frac{\partial g_j(\bs{R})}{\partial R_j}$ is a continuous function on $A\cap \mathbb R^n_{\geqslant}$,
where in the points with $R_j=0$ we take a~one-sided derivative. Similarly,  $\psi_j(\bs{R}^0)$ is~a~continuous 
function~on~ $A\cap \mathbb R^n_{\geqslant}$.

For $R_j\ne 0$ and $R_k\ne 0$, $k\ne j$, we have 
\begin{gather*}
\begin{pmatrix}\frac{\partial f_j}{\partial u_k}&\frac{\partial f_j}{\partial v_k}\end{pmatrix}=
\begin{pmatrix}\cos\varphi_k\left(\frac{\partial g_j}{\partial R_k}+
i  g_j\frac{\partial \psi_j}{\partial R_k}\right)e^{i(\varphi_j+\psi_j)}&
\sin\varphi_k\left(\frac{\partial g_j}{\partial R_k}+
i  g_j\frac{\partial \psi_j}{\partial R_k}\right)e^{i(\varphi_j+\psi_j)}
\end{pmatrix}
\end{gather*}
Also $\begin{pmatrix}\frac{\partial f_j}{\partial u_k}&\frac{\partial f_j}{\partial v_k}\end{pmatrix}\to
\begin{pmatrix}\frac{\partial f_j}{\partial u_k}(\bs{z}^0)&\frac{\partial f_j}{\partial v_k}(\bs{z}^0)\end{pmatrix}=
\begin{pmatrix}\bs{0}&\bs{0}\end{pmatrix}$ 
when $\bs{z}\to\bs{z}^0$.  Put $\varphi_1=0$, $\dots$, $\varphi_n=0$.
Under this condition $\bs{z}=\bs{R}\to \bs{R}^0$, and  
\begin{gather*}
\begin{pmatrix}\frac{\partial f_j}{\partial u_k}&\frac{\partial f_j}{\partial v_k}\end{pmatrix}=
\begin{pmatrix}\left(\frac{\partial g_j}{\partial R_k}+
i  g_j\frac{\partial \psi_j}{\partial R_k}\right)e^{i \psi_j}
&\bs{0}\end{pmatrix}\to\begin{pmatrix}\bs{0},\bs{0}\end{pmatrix}
\end{gather*}
Then $\frac{\partial g_j}{\partial R_k}+
i g_j\frac{\partial \psi_j}{\partial R_k}\to \bs{0}$ when $\bs{R}\to\bs{R}^0$ and $R_j, R_k\ne 0$. 
In particular,  $g_j(\bs{R})\frac{\partial \psi_j(\bs{R})}{\partial R_k}\to \bs{0}$, 
and 
$$
\frac{\partial g_j(\bs{R})}{\partial R_k}\to\bs{0}=\begin{cases}\frac{\partial g_j}{\partial R_k}(\bs{R}^0),
&\text{ if }k\notin \omega_0;\\
\frac{\partial g_j}{\partial_+R_k}(\bs{R}^0),&\text{ if }k\in \omega_0\setminus\{j\}.
\end{cases}
$$
For $u_j\ne 0$ the function $h_j(\bs{u})=g_j(|u_1|,\dots,|u_n|)=|f_j(u_1,\dots,u_n)|$ 
is smooth. Moreover,  
$$
\frac{\partial h_j(\bs{u})}{\partial u_k}=\begin{cases}
\frac{\partial g_j(|u_1|,\dots,|u_n|)}{\partial R_k},&\text{ if }u_k>0;\\
-\frac{\partial g_j(|u_1|,\dots,|u_n|)}{\partial R_k},&\text{ if }u_k<0,
\end{cases}
$$
and for $u_k=0$
\begin{gather*}
\frac{\partial h_j(\bs{u})}{\partial u_k}=\frac{\partial h_j(\bs{u})}{\partial_+u_k}=
\lim_{t\to 0^+}\frac{h_j(u_1,\dots, t,\dots, u_n)-h_j(u_1,\dots, 0,\dots, u_n)}{t}=\\
\lim_{t\to 0^+}\frac{g_j(|u_1|,\dots, t,\dots, |u_n|)-g_j(|u_1|,\dots, 0,\dots, |u_n|)}{t}=
\frac{\partial g_j(|u_1|,\dots,|u_n|)}{\partial_+R_k}=\\
=\frac{\partial h_j(\bs{u})}{\partial_-u_k}=
\lim_{t\to 0^-}\frac{h_j(u_1,\dots, t,\dots, u_n)-h_j(u_1,\dots, 0,\dots, u_n)}{t}=\\
\lim_{t\to 0^-}\frac{g_j(|u_1|,\dots, -t,\dots, |u_n|)-g_j(|u_1|,\dots, 0,\dots, |u_n|)}{t}=
-\frac{\partial g_j(|u_1|,\dots,|u_n|)}{\partial_+R_k}
\end{gather*}
Hence, $\frac{\partial g_j(\bs{R})}{\partial_+R_k}=0$ when $R_k=0$ and $R_j\geqslant 0$. 

This proves that $\frac{\partial g_j(\bs{R})}{\partial R_k}$, $k\ne j$, is a continuous function on $A\cap \mathbb R^n_{\geqslant}$,
where in the points with $R_k=0$ we take a~one-sided derivative (and it is equal to $0$).


Now consider $\w{f}_j(u_1,\dots,u_n)={\rm sign}(u_j) g_j(|u_1|,\dots,|u_n|)$ at $\bs{u}^0$. 
It is a continuous function, since $g_j(\bs{R})=0$ for $R_j=0$ and
$$
\w{f}_j(\bs{u})=\begin{cases}g_j(|u_1|,\dots,|u_n|),&\text{ if } u_j\geqslant 0;\\
-g_j(|u_1|,\dots,|u_n|),&\text{ if }u_j\leqslant 0;\\
0,&\text{ if }u_j=0.
\end{cases}
$$
For $k\ne j$ we have $\frac{\partial \w{f}_j(\bs{u})}{\partial u_k}=0$ if $u_j=0$, 
in particular for $\bs{u}^0$.
Moreover, for any~$u_j\geqslant 0$  
$$
\frac{\partial\w{f}_j(\bs{u})}{\partial u_k}={\rm sign}(u_j){\rm sign}(u_k)\frac{\partial g_j(|u_1|,\dots,|u_n|)}{\partial R_k}=
\begin{cases}
{\rm sign}(u_j)\frac{\partial g_j(|u_1|,\dots,|u_n|)}{\partial R_k},
&\text{ if }u_k>0;\\
-{\rm sign}(u_j)\frac{\partial g_j(|u_1|,\dots,|u_n|)}{\partial R_k},
&\text{ if }u_k<0;\\
0=\frac{\partial g_j(|u_1|,\dots,|u_n|)}{\partial_+R_k},
&\text{ if }u_k=0.
\end{cases}
$$
Hence, $\frac{\partial\w{f}_j(\bs{u})}{\partial u_k}$ is a continuous function at $\bs{u}^0$.
At last, for $k=j$ we have 
$$
\frac{\partial\w{f}_j(\bs{u})}{\partial u_j}=
\begin{cases}
\frac{\partial g_j(|u_1|,\dots,|u_n|)}{\partial R_j},
&\text{ if }u_j\ne 0;\\
\frac{\partial g_j(|u_1|,\dots,|u_n|)}{\partial_+R_j}>0,
&\text{ if }u_j=0.
\end{cases}
$$
Thus, $\frac{\partial\w{f}_j(\bs{u})}{\partial u_k}$ is also a continuous function at $\bs{u}^0$, and  
$\w{f}_j(\bs{u})$ is smooth at  $\bs{u}^0$. This finishes the proof. 
\end{proof}
Now Theorem~\ref{RPth} follows from Lemmas~\ref{RPl1} and~\ref{RPl2}. 

\end{proof}

\section{Acknowledgements}
The authors are grateful to V.M.~Buchstaber and A.A.~Gaifullin for fruitful discussions.


\begin{thebibliography}{99}
\bibitem[BM06]{BM06}
F.~Bosio, L.~Meersseman.
\emph{Real quadrics in $\mathbb{C}^n$, complex manifolds and convex polytopes}.
Acta Math., {\bf 197} (2006), 53--127.

\bibitem[BP02]{BP02} 
Victor M.~Buchstaber and Taras Panov. 
\emph{Torus Actions and Their Applications in Topology and Combinatorics}. 
University Lecture Series, 24. Amer. Math. Soc., Providence, RI, 2002.

\bibitem[BP15]{BP15}
Victor Buchstaber and  Taras Panov,
\emph{Toric Topology}. 
Math. Surv. and Monogr.,~204, Amer. Math. Soc., Providence, RI, 2015.

\bibitem[BPR07]{BPR07}
Victor Buchstaber, Taras Panov, and  Nigel Ray.  
\emph{Spaces of polytopes and cobordism of quasitoric manifolds}. 
Mosc. Math. J.~{\bf7} (2007), no.~2, 219--242.


\bibitem[D14]{D14} 
Michael W. Davis.
\emph{When Are Two Coxeter Orbifolds Diffeomorphic?}
Michigan Math. J. 63 (2014), 401--421.

\bibitem[D78]{D78}
M.~Davis.
\emph{Smooth $G$-manifolds as collections of fiber bundles}. 
Pacific J. Math. 77 (1978), no. 2, 315--363.


\bibitem[DJ91]{DJ91}
M.W.~Davis and T.~Januszkiewicz,\, 
\emph{Convex polytopes, Coxeter orbifolds and torus actions}, 
Duke Math. J.~{\bf62} (1991), no.~2, 417--451.

\bibitem[G08]{G08}
A.A.~Gaifullin.
\emph{Combinatorial realisation of cycles} (in russian).
PhD thesis, Lomonosov Moscow State University, 2008.


\bibitem[J12]{J12}
Dominic Joyce.
\emph{On manifolds with corners} 
Advances in geometric analysis, 225--258, Adv. Lect. Math. (ALM), 21, Int. Press,
Somerville, MA, 2012. ArXiv:0910.3518 [math.DG].

\bibitem[KMY18]{KMY18}
S.~Kuroki, M.~Masuda, L.~Yu
\emph{Small covers, infra-solvmanifolds and curvature},
arXiv:1111.2174v2 (2018).


\bibitem[W13]{W13}
M.~Wiemeler.
\emph{Exotic torus manifolds and equivariant smooth structures on quasitoric manifolds}. 
Math. Z. 273 (2013), 1063--1084, ArXiv:1110.1168.



\end{thebibliography}
\end{document}